\documentclass[a4paper, 11pt]{amsart}   
\usepackage{mathptmx, amssymb,amscd,latexsym, eulervm}   
\usepackage{amsmath, mathtools} 
\usepackage{amsthm}
\usepackage{mathdots}
\usepackage{hyperref}
\hypersetup{pagebackref,colorlinks=true,linkcolor=blue,urlcolor=blue,breaklinks=true}
\usepackage{color}
\usepackage[onehalfspacing]{setspace}
\usepackage{tabularx}
\usepackage{amsfonts}
\usepackage{paralist}
\usepackage{aliascnt}
\usepackage[initials, lite]{amsrefs}
\usepackage{amscd}
\usepackage{blkarray}
\usepackage{mathbbol}
\usepackage{setspace}
\usepackage[inner=2.4cm,outer=2.4cm, bottom=3.2cm]{geometry}
\usepackage{tikz, tikz-cd}
\usepackage{calligra,mathrsfs}

\usepackage{tikz}
\usetikzlibrary{matrix}
\usetikzlibrary{arrows,calc}
\allowdisplaybreaks

\BibSpec{collection.article}{%
	+{}  {\PrintAuthors}                {author}
	+{,} { \textit}                     {title}
	+{.} { }                            {part}
	+{:} { \textit}                     {subtitle}
	+{,} { \PrintContributions}         {contribution}
	+{,} { \PrintConference}            {conference}
	+{}  {\PrintBook}                   {book}
	+{,} { }                            {booktitle}
	+{,} { }                            {series}
	+{, vol.} { }                            {volume}
	+{,} { }                            {publisher}
	+{,} { \PrintDateB}                 {date}
	+{,} { pp.~}                        {pages}
	+{,} { }                            {status}
	+{,} { \PrintDOI}                   {doi}
	+{,} { available at \eprint}        {eprint}
	+{}  { \parenthesize}               {language}
	+{}  { \PrintTranslation}           {translation}
	+{;} { \PrintReprint}               {reprint}
	+{.} { }                            {note}
	+{.} {}                             {transition}
	+{}  {\SentenceSpace \PrintReviews} {review}
}
\AtBeginDocument{%
	\def\MR#1{}
}

\makeatletter
\@namedef{subjclassname@2020}{%
	\textup{2020} Mathematics Subject Classification}
\makeatother

\newcommand{\sA}{A}
\newcommand{\sB}{B}

\newcommand{\sG}{\mathscr{G}}

\newcommand{\sI}{\mathscr{I}}

\newcommand{\sM}{\mathscr{M}}

\newcommand{\CC}{\mathbb{C}}

\newcommand{\length}{{\rm length}}
\newcommand{\NN}{\normalfont\mathbb{N}}
\newcommand{\ZZ}{\mathbb{Z}}

\newcommand{\PP}{{\normalfont\mathbb{P}}}
\newcommand{\xx}{{\normalfont\mathbf{x}}}

\newcommand{\mm}{{\normalfont\mathfrak{m}}}
\newcommand{\init}{{\normalfont\text{in}}}

\newcommand{\QQ}{\mathbb{Q}}
\newcommand{\pp}{{\normalfont\mathfrak{p}}}
\newcommand{\br}{{\rm br}}

\newcommand{\qqq}{\mathfrak{q}}

\newcommand{\fP}{{\mathfrak{P}}}

\newcommand{\Ker}{\normalfont\text{Ker}}

\newcommand{\Ann}{\normalfont\text{Ann}}
\newcommand{\Supp}{\normalfont\text{Supp}}
\newcommand{\Ass}{{\normalfont\text{Ass}}}

\newcommand{\Min}{{\rm Min}}
\newcommand{\Sym}{\normalfont\text{Sym}}
\newcommand{\Rees}{\mathscr{R}}

\newcommand{\OO}{\mathcal{O}}

\newcommand{\HL}{\normalfont\text{H}_{\mm}}
\newcommand{\HH}{\normalfont\text{H}}

\newcommand{\iniTerm}{\normalfont\text{in}}

\newcommand{\Proj}{\normalfont\text{Proj}}
\newcommand{\Hilb}{{\normalfont\text{Hilb}}}
\newcommand{\Spec}{\normalfont\text{Spec}}

\newcommand{\biProj}{{\normalfont\text{BiProj}}}

\newcommand{\widesim}[2][1.2]{
	{\scalebox{#1}[1]{$\sim$}}
}

\DeclareMathOperator{\HT}{ht}

\DeclareMathOperator{\gr}{gr}

\def\f0{\mathbf{0}}

\def\ff{\mathbf{f}}

\def\1{\mathbf{1}}

\newtheorem{theorem}{Theorem}[section]

\newtheorem{headthm}{Theorem}

\newaliascnt{headcor}{headthm}

\aliascntresetthe{headcor}

\newaliascnt{headconj}{headthm}

\aliascntresetthe{headconj}

\newaliascnt{corollary}{theorem}
\newtheorem{corollary}[corollary]{Corollary}
\aliascntresetthe{corollary}

\newaliascnt{claim}{theorem}

\aliascntresetthe{claim}

\newaliascnt{lemma}{theorem}
\newtheorem{lemma}[lemma]{Lemma}
\aliascntresetthe{lemma}

\newaliascnt{conjecture}{theorem}

\aliascntresetthe{conjecture}

\newaliascnt{proposition}{theorem}
\newtheorem{proposition}[proposition]{Proposition}
\aliascntresetthe{proposition}

\theoremstyle{definition}
\newaliascnt{definition}{theorem}
\newtheorem{definition}[definition]{Definition}
\aliascntresetthe{definition}

\newaliascnt{notation}{theorem}

\aliascntresetthe{notation}

\newaliascnt{example}{theorem}

\aliascntresetthe{example}

\newaliascnt{examples}{theorem}

\aliascntresetthe{examples}

\newaliascnt{remark}{theorem}
\newtheorem{remark}[remark]{Remark}
\aliascntresetthe{remark}

\newaliascnt{question}{theorem}
\newtheorem{question}[question]{Question}
\aliascntresetthe{question}

\newaliascnt{questions}{theorem}

\aliascntresetthe{questions}

\newaliascnt{problem}{theorem}

\aliascntresetthe{problem}

\newaliascnt{construction}{theorem}

\aliascntresetthe{construction}

\newaliascnt{setup}{theorem}
\newtheorem{setup}[setup]{Setup}
\aliascntresetthe{setup}

\newaliascnt{setupdef}{theorem}

\aliascntresetthe{setupdef}

\newaliascnt{algorithm}{theorem}

\aliascntresetthe{algorithm}

\newaliascnt{observation}{theorem}

\aliascntresetthe{observation}

\newaliascnt{defprop}{theorem}

\aliascntresetthe{defprop}

\def\equationautorefname~#1\null{(#1)\null}
\def\sectionautorefname~#1\null{Section #1\null}
\def\subsectionautorefname~#1\null{\S #1\null}

\def\surjects{\twoheadrightarrow}

\DeclareFontFamily{OT1}{pzc}{}
\DeclareFontShape{OT1}{pzc}{m}{it}{<-> s * [1.100] pzcmi7t}{}
\DeclareMathAlphabet{\mathchanc}{OT1}{pzc}{m}{it}

\DeclareMathOperator{\fSpec}{\mathchanc{Spec}}
\DeclareMathOperator{\fProj}{\mathchanc{Proj}}

\title{Polar multiplicities and integral dependence}

\author{Yairon Cid-Ruiz}
\address{Department of Mathematics, North Carolina State University, Raleigh, NC 27695, USA}
\email{ycidrui@ncsu.edu}

\date{\today}

\begin{document}

	\maketitle

	\begin{abstract}
	We provide new criteria for the integrality and birationality of an extension of graded algebras in terms of the general notion of polar multiplicities of Kleiman and Thorup. 
	As an application, we obtain a new criterion for when a module is a reduction of another in terms of certain mixed  Buchsbaum-Rim multiplicities. 
	Furthermore, we prove several technical results regarding polar multiplicities.
	\end{abstract}

\section{Introduction}

The main goal of this paper is to provide multiplicity-based criteria for when an inclusion of graded algebras is integral. 
Our criteria are expressed in terms of the general notion of \emph{polar multiplicities} of Kleiman and Thorup \cite{KLEIMAN_THORUP_GEOM, KLEIMAN_THORUP_MIXED}.

\medskip

Ever since the seminal work of Rees \cite{Rees1961}, the search for numerical criteria to characterize \emph{integral dependence} has been an important research topic in algebraic geometry, commutative algebra and singularity theory.
Rees showed that, in an equidimensional and universally catenary Noetherian local ring $(R,\mm)$, two $\mm$-primary ideals $I \subseteq J$ have the same integral closure if and only if they have the same Hilbert-Samuel multiplicity. 
Considerable effort has been made to extend this classical result to the case of arbitrary ideals, modules, and, more generally, algebras (see \cite{KLEIMAN_THORUP_GEOM,KLEIMAN_THORUP_MIXED, UV_NUM_CRIT, UV_CRIT_MOD, BOEGER, FLENNER_MANARESI, GAFFNEY, GG, PTUV}). 
Notably, a complete extension of Rees' theorem for arbitrary ideals was recently obtained by Polini, Trung, Ulrich and Validashti \cite{PTUV} in terms of the multiplicity sequence of Achilles and Manaresi \cite{ACHILLES_MANARESI_MULT}.
When $R = \OO_{X,0}$ and  $(X, 0) \subset (\CC^N, 0)$ is a reduced closed analytic subspace of pure dimension,  this extension to arbitrary ideals was previously settled by Gaffney and Gassler \cite{GG} in terms of Segre numbers. 

\medskip

Let $(R,\mm)$ be a Noetherian local ring and $A \subseteq B$ be a homogeneous inclusion of standard graded algebras over $R = A_0 = B_0$.
Let $X = \Proj(B)$, $Y = \Proj(A)$ and $r = \dim(X)$, and consider the associated morphism 
$$
f \;:\; U \subseteq X \rightarrow Y \quad  \text{ where } \quad U = X \setminus V_+(A_+B).
$$
Following Kleiman and Thorup \cites{KLEIMAN_THORUP_GEOM,KLEIMAN_THORUP_MIXED}, the polar multiplicities of $B$ are defined by the following intersection numbers
$$
m_r^i(B) \;=\; \int c_1\left(\OO_D(1)\right)^i c_1\left(\OO_X(1)\right)^{r-i}\, \big[D\big]_r \quad \text{ for all } \quad 0 \le i \le r;
$$
here $D$ denotes the projective completion of the normal cone $C_ZX$ to  $Z$ in $X$, where $Z \subseteq X$ is the preimage of the closed point of $\Spec(R)$.
Motivated by the work of Simis, Ulrich and Vasconcelos \cite{SUV_MULT}, we utilize polar multiplicities of the intermediate algebra $G=\gr_{A_+B}(B)$ to study the integrality and birationality of the extension $A \hookrightarrow B$.
We see $G$ as a standard graded $R$-algebra (see \autoref{sect_criteria}), and we consider the following polar multiplicities 
$$
m_r^i(A, B) \;=\; m_r^i(G) \quad \text{ for all } \quad 0 \le i \le r.
$$
Our main result is the following theorem which provides criteria for integrality and birationality in terms of the polar multiplicities of the algebras $A$, $B$ and $G$.

\begin{headthm}[{\autoref{thm_criteria}}]
	\label{thmA}
	Under the notations and assumptions above, the following statements hold: 
	\begin{enumerate}[\rm (i)]
	\item {\rm(Integrality)} Consider the following conditions: 
	\begin{enumerate}[\rm (a)]
		\item A finite morphism $f : X \rightarrow Y$ is obtained.
		\item $m_r^i(A, B) = m_r^i(G/B_tG)$ for all $0 \le i \le r$ and $t \gg 0$.
		\item $m_r^i(A, B) = m_r^i(B)$ for all $0 \le i \le r$.
	\end{enumerate}
	In general, the implications {\rm(a)} $\Rightarrow$ {\rm(b)} and {\rm(a)} $\Rightarrow$ {\rm(c)} hold. 
	Moreover, if $B$ is equidimensional and catenary, then the reverse implication {\rm(b)} $\Rightarrow$ {\rm(a)} also holds.
	\item {\rm(Integrality $+$ Birationality)} If we obtain a finite birational morphism $f : X \rightarrow Y$, then $m_r^i(A, B) = m_r^i(A)$ for all $0 \le i \le r$.
	The converse holds if $B$ is equidimensional and catenary.
	\end{enumerate}
\end{headthm}	

Notice that the above theorem yields the equalities $m_r^i(A) = m_r^i(A, B) = m_r^i(B)$ for all $0 \le i \le r$ if we obtain a finite birational morphism $f : X \rightarrow Y$.
Therefore we should ask the following question:
 
 \begin{question}
	If we have $m_r^i(A) = m_r^i(B)$ for all $0 \le i \le r$,  under which conditions do we obtain a finite birational morphism $f : X \rightarrow Y$?
 \end{question}
 
This question has already a positive answer in the case $R$ is zero-dimensional. 
If $R$ is Artinian, the only non-vanishing polar multiplicity is $m_r^0$ and it coincides with the usual multiplicity of a graded algebra (see \autoref{prop_gen_polar}(i) and \autoref{lem_vanish}).
Therefore, \cite[Proposition 6.1]{SUV_MULT}  yields a positive answer under the conditions that $\dim(R) = 0$ and $B$ is equidimensional.

\medskip

Given a finitely generated $R$-module $E$, we can study a notion of \emph{mixed Buchsbaum-Rim multiplicities} of $E$ by considering the polar multiplicities of the Rees algebra $\Rees(E)$ of $E$.
In \autoref{thm_reduction_mod}, we provide a criterion for integral dependence of modules in terms of these mixed Buchsbaum-Rim multiplicities.

\medskip

Finally, we also provide a number of technical results on polar multiplicities that might be interesting in their own right:
\begin{enumerate}[\rm (i)]
	\item \autoref{thm_additive_polar} gives a precise description of the additive behavior of polar multiplicities. 
	\item \autoref{thm_hypersurface_sect} describes the behavior of polar multiplicities under hypersurface sections. 
	\item \autoref{lem_localiz_polar} shows that, under mild conditions,  polar multiplicities do not increase under localizations. 
	\item In \autoref{sect_length_formula}, we relate polar multiplicities to certain St\"uckrad-Vogel cycles and by using a deformation to the normal cone we obtain a length formula. 
\end{enumerate}

\medskip
\noindent
{\bf Outline.}
The basic outline of this paper is as follows.
In \autoref{sect_basic}, we prove a number of basic results regarding polar multiplicities. 
By using a technique of deformation to the normal cone, we give a length formula for polar multiplicities in \autoref{sect_length_formula}.
The proof of \autoref{thmA} is given in \autoref{sect_criteria}.
Finally, in \autoref{sect_modules}, we apply \autoref{thmA} to obtain a criterion for integral dependence of modules in terms of certain mixed Buchsbaum-Rim multiplicities. 

\section{Basic results on polar multiplicities}
\label{sect_basic}

In this section, we recall the notion of polar multiplicities as defined by Kleiman and Thorup \cite{KLEIMAN_THORUP_GEOM,KLEIMAN_THORUP_MIXED}, and we prove some basic results for these invariants. 
Our main goal is to extend some of the results in the elegant treatment of Flenner--O'Carroll--Vogel \cite[\S 1.2]{FLENNER_O_CARROLL_VOGEL} (of Hilbert-Samuel multiplicities) to the case of polar multiplicities. 
The setup below is used throughout this section.

\begin{setup}
	\label{setup_basic_polar}
	Let $(R, \mm, \kappa)$ be a Noetherian local ring with maximal ideal $\mm$ and residue field $\kappa$.
	Let $B$ be a standard graded algebra over $R = B_0$.
	Set $X := \Proj(B)$.
\end{setup}

Following the work of Kleiman and Thorup \cites{KLEIMAN_THORUP_GEOM,KLEIMAN_THORUP_MIXED}, we define \emph{polar multiplicities} as follows. 
Let $M$ be a finitely generated graded $B$-module and $\sM := \widetilde{M}$ be the corresponding coherent sheaf on $X$.
Let $p : X \rightarrow \Spec(R)$ be the natural projection. 
Take $Z := p^{-1}(\{\mm\}) \subset X$ to be preimage of the closed point of $\Spec(R)$ and $\mathscr{I}_Z \subset \OO_X$ be its ideal sheaf.
Consider the corresponding normal cone $C:= C_ZX =\fSpec_X\left(\bigoplus_{n=0}^\infty\mathscr{I}_Z^n/\mathscr{I}_Z^{n+1}\right)$ and its projective completion
$$
D \;:=\; \PP\left(C \oplus 1\right) \;=\; \fProj_{X}\Bigg(\Big(\bigoplus_{n=0}^\infty\mathscr{I}_Z^n/\mathscr{I}_Z^{n+1}\Big)[z]\Bigg);
$$
here $z$ is a new variable of degree one, see, e.g., \cite[Appedinx B.5]{FULTON_INTERSECTION_THEORY}.
Denote by $C_\sM$  the coherent sheaf on $C$ determined by the module $\bigoplus_{n=0}^\infty \sI_Z^n\sM/\sI_Z^{n+1}\sM$.
Denote by $D_\sM$ the coherent sheaf on $D$ determined by the module $\big(\bigoplus_{n=0}^\infty \sI_Z^n\sM/\sI_Z^{n+1}\sM\big)[z]$.

\begin{definition}
	For all $r \ge \dim(\Supp(\sM))$ and $0 \le i \le r$, we say that the corresponding \emph{polar multiplicity} of $M$ is given by the intersection number 
	$$
	m_r^i(M) \;:= \; \int c_1\left(\OO_D(1)\right)^i c_1\left(\OO_X(1)\right)^{r-i}\, \big[D_\sM\big]_r.
	$$
	We also set $m^i(M) := m_{r}^i(M)$ with $r = \dim(\Supp(\sM))$.
\end{definition}

We now express the above invariants as normalized leading coefficients of a certain bivariate Hilbert polynomial.
Consider the following associated graded ring and associated graded module
$$
\sG  := \gr_{\mm B}(\sB) \;=\; \bigoplus_{n=0}^\infty \,\mm^n\sB\big/ \mm^{n+1}\sB \quad \text{ and } \quad \sG_M  := \gr_{\mm B}(M) \;=\; \bigoplus_{n=0}^\infty \,\mm^nM\big/ \mm^{n+1}M
$$
with respect to the ideal $\mm \sB$.
Notice that $\sG = \gr_{\mm B}(\sB)$ is a standard $\NN^2$-graded algebra over the residue field $\kappa$ and that $\sG_M$ is a finitely generated bigraded $\sG$-module.
For any $v \in \ZZ$ and $n \in \NN$, the corresponding $(v,n)$-graded part is given by $\big[\sG_M\big]_{(v,n)} = \mm^nM_v \big/ \mm^{n+1}M_v$.
We also study the following completed objects
$$
\sG^\star :=  \gr_{(\mm,t)}(\sB[t]) \;=\; \bigoplus_{n=0}^\infty \frac{(\mm,t)^n\sB[t]}{(\mm,t)^{n+1}\sB[t]} \quad \text{ and } \quad \sG_M^\star :=  \gr_{(\mm,t)}(M[t]) \;=\; \bigoplus_{n=0}^\infty \frac{(\mm,t)^nM[t]}{(\mm,t)^{n+1}M[t]}
$$ 
where $S := R[t]$ with $t$ a new variable, $B[t]:= B \otimes_R S$ and $M[t] := M \otimes_R S$.
Notice that we obtain an isomorphism of bigraded $\sG^\star$-modules
\begin{equation}
	\label{eq_ext_ass_gr}
	\sG_M^\star \;=\; \gr_{(\mm,t)}(M[t]) \;\cong\; \gr_{\mm B}(M)[t^\star]	
\end{equation}
where $t^\star$ is a new variable of degree $(0,1)$.
Hence, we obtain the equality 
$$
\length_R\left(\left[\sG_M^\star\right]_{(v,n)}\right) \;=\; \length_R\left(M_v/\mm^{n+1}M_v\right).
$$

Let $r := \dim(\Supp(\sM))$.
The bivariate Hilbert polynomials of $\sG_M$ and $\sG_M^\star$ are denoted by $P_{\sG_M}(v, n)$ and $P_{\sG_M^\star}(v, n)$, respectively.
We start with the following degree computations.

\begin{lemma}
	\label{lem_dim_sG_star}
	 We have  $\deg(P_{\sG_M}) \le r-1$ and $\deg(P_{\sG_M^\star})=r$.
\end{lemma}
\begin{proof}
	We may substitute $M$ by $\overline{M} = M/\HH_{B_+}^0(M)$ and $B$ by $\overline{B} = B/\Ann_B(M)$.
	Thus may assume that $\dim(B) = \dim(M) = r+1$ and that $\dim(R) = \dim(B/B_+) < r+1$.
	It then follows that $\deg(P_{\sG_M}) \le \dim(\sG_M) -2 =r-1$ and $\deg(P_{\sG_M^\star}) \le \dim(\sG_M^\star)- 2 = r$.
	To prove the equality $\deg(P_{\sG_M^\star}) = r$, it suffices to check that no minimal prime of $\sG_M^\star$ of dimension $r+2$ can contain $[\sG^\star]_{(1,0)} \cap [\sG^\star]_{(0,1)}$.
	Notice that no such minimal prime can contain $[\sG^\star]_{(0,1)}$ due to \autoref{eq_ext_ass_gr}, and that no such minimal prime can contain $[\sG^\star]_{(1,0)}$ because $\dim(\sG^\star/([\sG^\star]_{(1,0)}))=\dim(\gr_\mm(R)[t^\star]) = \dim(R)+1 < r+2$.
	So the result follows.
\end{proof}
Then we write
$$
P_{\sG_M}(v, n) \;=\; \sum_{i+j=r-1}\frac{e\left(i,j;\,\sG_M\right)}{i!\,j!}\,v^in^j \;+\; \text{(lower degree terms)}
$$ 
and 
$$
P_{\sG_M^\star}(v, n) \;=\; \sum_{i+j=r}\frac{e(i,j;\,\sG_M^\star)}{i!\,j!}\,v^in^j \;+\; \text{(lower degree terms)}.
$$
We have the following connections between the polar multiplicities of $M$ and the bivariate polynomials $P_{\sG_M}$ and $P_{\sG_M^*}$.

\begin{lemma}
	\label{lem_basic_polar}
	The following statements hold: 
	\begin{enumerate}[\rm (i)]
		\item $m_r^i(M) = e\left(r-i, i; \sG_M^\star\right)$ for all $0 \le i \le r$.
		\item $e(i, j+1; \sG_M^\star) = e(i,j;\sG_M)$ for all $i + j = r-1$.
	\end{enumerate}
\end{lemma}
\begin{proof}
	(i) This is a direct consequence of \cite[Lemma 4.3]{KLEIMAN_THORUP_GEOM
	}.
	
	(ii) We have the equality $\Hilb_{\sG_M^\star}(t_1,t_2) = \frac{1}{1-t_2}\Hilb_{\sG_M}(t_1,t_2)$ of Hilbert series (see \autoref{eq_ext_ass_gr}).
	Then the result follows from \cite[Theorem A]{MIXED_MULT}.
\end{proof}

We consider the following \emph{Rees algebra} and \emph{Rees module} 
$$
\Rees(\mm; B) \;:=\; \bigoplus_{n=0}^\infty \mm^nBT^n \subset B[T] \quad \text{ and } \quad  
\Rees(\mm; M) \::=\; \bigoplus_{n=0}^\infty \mm^nMT^n \subset M[T].
$$
We have that $\Rees(\mm; B)$ is naturally a standard bigraded $R$-algebra and that  $\Rees(\mm; M)$ is finitely generated bigraded module over $\Rees(\mm; B)$.
We also consider the \emph{extended Rees algebra} 
$$
\Rees^+(\mm; B) \;:=\; B[\mm T, T^{-1}] \;=\; \bigoplus_{n\in \ZZ} \mm^nBT^n \;\subset \; B[T, T^{-1}]
$$ and the \emph{extended Rees module} 
$$
\Rees^+(\mm; M) \;:=\; \bigoplus_{n \in \ZZ} \mm^nMT^n \;\subset\; M[T, T^{-1}].
$$
Recall that we have the isomorphism 
$$
\sG_M \;=\; \gr_{\mm B}(M) \;\cong\; \Rees^+(\mm; M) \big/ T^{-1} \Rees^+(\mm; M).
$$

The following theorem is inspired by \cite[Theorem 1.2.6]{FLENNER_O_CARROLL_VOGEL}.
Here, we obtain a quite precise description of the behavior of polar multiplicities under short exact sequences.

\begin{theorem}
	\label{thm_additive_polar}
	Assume \autoref{setup_basic_polar}.
	Let $0 \rightarrow M' \rightarrow M \rightarrow M'' \rightarrow 0$ be a short exact sequence of finitely generated graded $\sB$-modules. 
	Set $r = \dim(\Supp(\widetilde{M}))$.
	Then the following statements hold:
	\begin{enumerate}[\rm (i)]
		\item The cycles on $\biProj\left(\sG\right)$ associated to the bigraded modules
		$$
		\Ker\big(\sG_{M'} \rightarrow \sG_M\big) \qquad \text{ and } \qquad \Ker\big(\sG_M/\sG_{M'} \rightarrow \sG_{M''}\big)
		$$
		are equal.
		\item Let $h$ be the dimension of the supports of the sheaves on $\biProj(\sG)$ associated to the kernels in part {\rm (i)}.
		Then, for $v \in \ZZ$ and $n \ge 0$,
		$$
		Q(v, n) \;:=\; \length_R\big(\left[\sG_{M'}^\star\right]_{(v, n)}\big) + \length_R\big(\left[\sG_{M''}^\star\right]_{(v, n)}\big) - \length_R\big(\left[\sG_{M}^\star\right]_{(v, n)}\big) \;\ge\; 0,
		$$
		and $Q(v,n)$ coincides with a polynomial of degree $h$ for $v \gg 0$ and $n \gg 0$.
		\item $m_r^i(M) = m_r^i(M') + m_r^i(M'')$
		for all $0 \le i \le r$.
	\end{enumerate}
	
\end{theorem}
\begin{proof}
	Let $N := \Ker\left(\Rees^+(\mm; M) \rightarrow \Rees^+(\mm; M'')\right)$ and $L := N / \Rees^+(\mm; M')$.
	We have the following commutative diagram with exact rows and columns 
	\begin{equation*}
		\begin{tikzpicture}[baseline=(current  bounding  box.center)]
			\matrix (m) [matrix of math nodes,row sep=3em,column sep=4.5em,minimum width=2em, text height=1.5ex, text depth=0.25ex]
			{
				& 0 & 0 & 0 &	\\			
				0 & N(0,1) & \Rees^+(\mm; M)(0,1) & \Rees^+(\mm; M'')(0,1) & 0\\
				0 & N & \Rees^+(\mm; M) & \Rees^+(\mm; M'') & 0.\\
			};
			\path[-stealth]
			(m-2-1) edge (m-2-2)
			(m-2-2) edge (m-2-3)
			(m-2-3) edge (m-2-4)
			(m-2-4) edge (m-2-5)
			(m-3-1) edge (m-3-2)
			(m-3-2) edge (m-3-3)
			(m-3-3) edge (m-3-4)
			(m-3-4) edge (m-3-5)
			(m-1-2) edge (m-2-2)
			(m-1-3) edge (m-2-3)
			(m-1-4) edge (m-2-4)
			(m-2-2) edge node [left] {$T^{-1}$} (m-3-2)
			(m-2-3) edge node [left] {$T^{-1}$} (m-3-3)
			(m-2-4) edge node [left] {$T^{-1}$} (m-3-4)
			;		
		\end{tikzpicture}	
	\end{equation*}
	The snake lemma yields the short exact sequence 
	\begin{equation}
		\label{eq_ex_seq_N_gr_gr}
		0 \longrightarrow N/T^{-1}N \longrightarrow \sG_M \longrightarrow \sG_{M''} \longrightarrow 0.
	\end{equation}
	Let $U$ and $V$ be the kernel and cokernel of the multiplication map $L(0,1) \xrightarrow{T^{-1}} L$, respectively.
	Hence we consider the exact sequence 
	\begin{equation}
		\label{eq_ex_seq_mult_L}
		0 \longrightarrow U \longrightarrow L(0,1) \xrightarrow{\quad T^{-1} \quad} L \longrightarrow V \longrightarrow 0.
	\end{equation}
	We obtain the following commutative diagram with exact rows and columns 
	\begin{equation*}
		\begin{tikzpicture}[baseline=(current  bounding  box.center)]
			\matrix (m) [matrix of math nodes,row sep=3em,column sep=4.5em,minimum width=2em, text height=1.5ex, text depth=0.25ex]
			{
				&  &  & 0 &	\\			
				& 0 & 0 & U &	\\			
				0 & \Rees^+(\mm; M')(0,1) &  N(0,1) & L(0,1) & 0\\
				0 & \Rees^+(\mm; M') &  N & L & 0\\
				&  &  & V &	\\
				&  &  & 0. &	\\
			};
			\path[-stealth]
			(m-3-1) edge (m-3-2)
			(m-3-2) edge (m-3-3)
			(m-3-3) edge (m-3-4)
			(m-3-4) edge (m-3-5)
			(m-4-1) edge (m-4-2)
			(m-4-2) edge (m-4-3)
			(m-4-3) edge (m-4-4)
			(m-4-4) edge (m-4-5)
			(m-2-2) edge (m-3-2)
			(m-2-3) edge (m-3-3)
			(m-2-4) edge (m-3-4)
			(m-3-2) edge node [left] {$T^{-1}$} (m-4-2)
			(m-3-3) edge node [left] {$T^{-1}$} (m-4-3)
			(m-3-4) edge node [left] {$T^{-1}$} (m-4-4)
			(m-1-4) edge (m-2-4)
			(m-4-4) edge (m-5-4)
			(m-5-4) edge (m-6-4)
			;		
		\end{tikzpicture}	
	\end{equation*}
	By utilizing the snake lemma, we get the exact sequence
	\begin{equation}
		\label{eq_ex_seq_gr_A_N}
		0 \longrightarrow U \longrightarrow \sG_{M'} \longrightarrow N/T^{-1}N \longrightarrow V \longrightarrow 0.
	\end{equation}
	From \autoref{eq_ex_seq_N_gr_gr} and \autoref{eq_ex_seq_gr_A_N}, we obtain
	\begin{equation}
		\label{eq_eq_U_V_kernels}
		U \;\cong\; \Ker\big(\sG_{M'} \rightarrow \sG_M\big) \qquad \text{ and } \qquad V \;\cong\; \Ker\big(\sG_M/\sG_{M'} \rightarrow \sG_{M''}\big).		
	\end{equation}
	Since $N$ and $\Rees^+(\mm; M')$ coincide for nonpositive powers of $T$, it follows that some power of $T^{-1}$ annihilates $L$.
	Hence, $L$ is a finitely generated bigraded module over $\Rees(\mm;B) / \mm^{k}\Rees(\mm;B)$ for some $k > 0$, and the latter is a standard bigraded algebra over the Artinian local ring $R/\mm^k$.
	Furthermore, if $\pp \in \Spec(\sG)$ is a minimal prime of $L$, then it is also a minimal prime of $U$ and $V$ and we get the equality $\length_{\sG_\pp}(U_\pp) = \length_{\sG_\pp}(V_\pp)$.
	So, the cycles associated to $U$ and $V$ coincide, and this settles part (i) by the isomorphisms in \autoref{eq_eq_U_V_kernels}.
	
	By combining \autoref{eq_ex_seq_N_gr_gr}, \autoref{eq_ex_seq_mult_L} and \autoref{eq_ex_seq_gr_A_N}, for all $v \in \ZZ$ and $n \ge 0$, we obtain the equalities 
	$$
		\length_R\big(\left[L\right]_{(v, n+1)}\big) =  \sum_{j=0}^n \length_R\big(\left[\sG_{M'}\right]_{(v,j)}\big) + \sum_{j=0}^n \length_R\big(\left[\sG_{M''}\right]_{(v,j)}\big) - \sum_{j=0}^n \length_R\big(\left[\sG_{M}\right]_{(v,j)}\big)
	$$
	and 
	$$
		\length_R\big(\left[L\right]_{(v, n+1)}\big) = \length_R\big(\left[\sG_{M'}^{\star}\right]_{(v, n)}\big) + \length_R\big(\left[\sG_{M''}^{\star}\right]_{(v, n)}\big) - \length_R\big(\left[\sG_{M}^\star\right]_{(v, n)}\big).
	$$
	The function $Q(v, n-1) = \length_{R}([L]_{(v,n)})$ eventually coincides with the bivariate Hilbert polynomial $P_L(v, n)$ of $L$. 
	Since $P_L(v, n)$ has degree equal to $h$, the result of part (ii) follows. 
	
	Since $h \le r-1$, the equality $m_r^i(M) = m_r^i(M') + m_r^i(M'')$ follows from part (ii) and \autoref{lem_basic_polar}(i).
	Therefore, the proof of the theorem is complete.
\end{proof}

A standard corollary of \autoref{thm_additive_polar} is an expected associativity formula for polar multiplicities. 

\begin{corollary}
	\label{cor_associative_polar}
	Let $M$ be a finitely generated graded $B$-module, and set $r = \dim(\Supp(\widetilde{M}))$.
	We have the equality $m_r^i(M) = \sum_\pp \length_{B_\pp}(M_\pp) m_r^i(B/\pp)$ where the sum runs through the minimal primes of $M$ of dimension $r+1$.
\end{corollary}
\begin{proof}
	We can choose a prime filtration of $M$ (see, e.g., \cite[Proposition 3.7]{EISEN_COMM}) and apply \autoref{thm_additive_polar}(iii).
\end{proof}

We now describe the behavior of polar multiplicities under hypersurface sections. 
For that purpose, it will be convenient to introduce the following notation. 
For any homogeneous element $b \in \mm^\beta B_\alpha \setminus \mm^{\beta+1} B_\alpha$, we denote by $\iniTerm(b) \in [\sG]_{(\alpha, \beta)}$ the initial form of $b$ in $\sG$.

\begin{definition}[{cf. \cite[Definition 1.2.10]{FLENNER_O_CARROLL_VOGEL}}]
	Let $b \in \mm^\beta B_\alpha \setminus \mm^{\beta+1} B_\alpha$ and set $r = \dim(\Supp(\widetilde{M}))$.
	We say that $b$ is a \emph{$\sG$-parameter on $M$} if we have the strict inequality 
	$$
	\dim\left(\Supp\left(\left(\sG_M/\iniTerm(b)\sG_M\right)^{\widesim{}}\right)\right) \; < \; r-1,
	$$
	where $\left(\sG_M/\iniTerm(b)\sG_M\right)^{\widesim{}}$ denotes the coherent sheaf on $\biProj(\sG)$ associated to $\sG_M/\iniTerm(b)\sG_M$.
	By convention, the dimension of the empty set is $-1$.
\end{definition}

The next theorem gives a sharp description of the behavior of polar multiplicities under hypersurface sections.
It extends the result of \cite[Theorem 1.2.11]{FLENNER_O_CARROLL_VOGEL} to the case of polar multiplicities. 

\begin{theorem}
\label{thm_hypersurface_sect}
Assume \autoref{setup_basic_polar}.
Let $M$ be a finitely generated graded $B$-module and $b \in \mm^\beta B_\alpha \setminus \mm^{\beta+1} B_\alpha$.
Set $r = \dim(\Supp(\widetilde{M}))$ and $N = (0:_Mb)$, and suppose that $\dim(\Supp((M/bM)^{\widesim{}}))<r$.
Then 
$$
m_{r-1}^i(M/bM) \;\ge\; \alpha m_r^i(M) + \beta m_r^{i+1}(M) + m_{r-1}^i(N) \quad \text{ for all \quad $0 \le i \le r-1$},
$$
and equality holds for all $0 \le i \le r-1$ if and only if $b$ is a $\sG$-parameter on $M$.
\end{theorem}
\begin{proof}
	Notice that we also have $\dim(\Supp(\widetilde{N})) < r$.
	By applying \autoref{thm_additive_polar} to the short exact sequences
	\begin{equation*}
		0 \rightarrow N(-\alpha) \rightarrow M(-\alpha) \xrightarrow{ \;\cdot b\;} bM \rightarrow 0 \quad \text{ and } \quad 
		0 \rightarrow bM \rightarrow \mm^\beta M \rightarrow \mm^\beta M / b M \rightarrow 0, 
	\end{equation*}
	we obtain polynomials 
	$$
	Q_1(v, n) =   P_{\sG_{N}^\star}(v-\alpha, n) + P_{\sG_{bM}^\star}(v, n) - P_{\sG_{M}^\star}(v-\alpha, n)
	$$
	and 
	$$
	Q_2(v, n) = P_{\sG_{bM}^\star}(v, n) +  P_{\sG_{\mm^\beta M/bM}^\star}(v, n) - P_{\sG_{\mm^\beta M}^\star}(v, n)
	$$
	with nonnegative leading coefficients such that 
	$$\deg(Q_1) = \dim(\Supp(\Ker(\sG_N \rightarrow \sG_M)^{\widesim{}})) \le \dim(\Supp(\widetilde{N})) - 1 < r-1
	$$ 
	and 
	\begin{align}
		\label{eq_deg_poly_Q2}
		\begin{split}
			\deg(Q_2) &= \dim\left(\Supp\left(
			\Ker\left(\sG_{\mm^\beta M}/\sG_{bM} \rightarrow \sG_{\mm^\beta M/bM}\right)^{\widesim{}}
			\right)\right) \\
			&= \dim\left(\Supp\left(
			\Ker\left(\sG_{M}/\iniTerm(b)\sG_{M} \rightarrow \sG_{M/bM}\right)^{\widesim{}}
			\right)\right) \le r-1.
		\end{split}
	\end{align}
	We have the equalities $P_{\sG_{\mm^\beta M}^\star}(v, n) = P_{\sG_{M}^\star}(v, n+\beta) - \length_R(M_v/\mm^\beta M_v)$ and $P_{\sG_{\mm^\beta M/bM}^\star}(v, n) = P_{\sG_{M/bM}^\star}(v, n+\beta) - \length_R(M_v/\mm^\beta M_v)$.
	Therefore, we obtain
	$$
	Q_2(v, n) - Q_1(v,n) = P_{\sG_{M/bM}^\star}(v, n+\beta) - \left(P_{\sG_M^\star}(v,n+\beta)-P_{\sG_M^\star}(v-\alpha,n)\right) - P_{\sG_N^\star}(v-\alpha,n).
	$$
	Since $\deg(Q_1) < r-1$ and $\deg(Q_2) \le r-1$, \autoref{rem_diff_poly} yields the claimed inequality 
	$$
	m_{r-1}^i(M/bM) \;\ge\; \alpha m_r^i(M) + \beta m_r^{i+1}(M) + m_{r-1}^i(N) \quad \text{ for all \quad $0 \le i \le r-1$}.
	$$
	Moreover, we have an equality for all $0 \le i \le r-1$ if and only if $\deg(Q_2) < r-1$.
	 Due to \autoref{eq_deg_poly_Q2}, the latter condition is equivalent to 
	 $$
	 \dim\left(\Supp\left(
	 \Ker\left(\sG_{M}/\iniTerm(b)\sG_{M} \rightarrow \sG_{M/bM}\right)^{\widesim{}}
	 \right)\right) < r-1.
	 $$
	 Finally, since $\dim(\Supp((\sG_{M/bM})^{\widesim{}}))<r-1$ by assumption, it follows that equality is attained for all $0 \le i \le r-1$ if and only if $b$ is a $\sG$-parameter on $M$.
\end{proof}

\begin{remark}
	\label{rem_diff_poly}
	Let $P(t_1,t_2) \in \QQ[t_1,t_2]$ be a polynomial of degree $r$ and write 
	$$
	P(t_1,t_2)  \, = \, \sum_{i+j=r} \frac{p_{i,j}}{i!j!} t_1^it_2^j \,+ \, \text{(lower degree terms)}.
	$$
	Let $\alpha, \beta \in \QQ$ and set $Q(t_1,t_2) = P(t_1,t_2) - P(t_1-\alpha,t_2-\beta)$.
	We then obtain that 
	$$
	Q(t_1,t_2) \, =  \, \sum_{i+j=r-1} \frac{\alpha p_{i+1,j} + \beta p_{i,j+1}}{i!j!} t_1^it_2^j \,+ \, \text{(lower degree terms)}.
	$$
\end{remark}

Next, we have a proposition that compiles some useful properties of polar multiplicities. 
We are particularly interested in the behavior of polar multiplicities under general hyperplane sections (i.e., cutting with general elements in $\mm B$ or in $B_+ = \bigoplus_{v \ge 1} B_v$).
We also desire more explicit descriptions of the two extremal polar multiplicities $m_r^0$ and $m_r^r$.
Suppose that $\kappa$ is an infinite field, and let $I \subset B$ be an ideal generated by homogeneous elements $f_1,\ldots,f_e \in B$ of the same degree $\delta \ge 0$.
We say that a property $\mathcal{P}$ holds for a \emph{general element} $f \in I_\delta$ if there exists a dense Zariski-open subset $U \subset \kappa^e$ such that whenever $f = a_1f_1+\cdots +a_ef_e$ and the image of $(a_1,\ldots,a_e)$ belongs to $U$, then the property $\mathcal{P}$ holds for $f$.
For a finitely generated graded $B$-module $M$, its \emph{$j$-multiplicity} is given by 
$j_d(M) := e_d(\HL^0(M))$ (see \cite[\S 6.1]{FLENNER_O_CARROLL_VOGEL} for more details on this notion).
The proposition below extends \cite[Proposition 8.2]{KLEIMAN_THORUP_GEOM}.

\begin{proposition}
	\label{prop_gen_polar}
	Let $M$ be a finitely generated graded $B$-module, and set $r = \dim(\Supp(\widetilde{M}))$.
	The following statements hold:
	\begin{enumerate}[\rm (i)]
		\item $m_r^0(M) = j_{r+1}(M)$.
		\item $m_r^r(M) = e_{r}\big(\HH^0(X, \widetilde{M})\big)$.
		\item {\rm(}$\kappa$ infinite{\rm)} Let $y \in B_+$ be a general element.
		Then, for all $0 \le i \le r-1$, $m_{r-1}^{i}(M/yM) = m_r^i(M)$.
		\item {\rm(}$\kappa$ infinite{\rm)} Assume  $\dim(\Supp((M/\mm M)^{\widesim{}})) < r$, and let $x \in \mm$ be a general element. 
		Then, for all $0 \le i \le r-1$, $m_{r-1}^i(M/xM) = m_{r}^{i+1}(M) + m_{r-1}^i((0:_Mx))$.
	\end{enumerate}
\end{proposition}
\begin{proof}
	(i) For $n \gg 0$ big enough, we obtain the equalities
	$$
	m^0_r(M) \;=\; \lim_{v \rightarrow \infty} \frac{P_{\sG_M^\star}(v,n)}{v^r/r!} \;=\; \lim_{v \rightarrow \infty} \frac{\length_{R}(M_v/\mm^{n+1}M_v)}{v^r/r!} \;=\; e_{r+1}(M/\mm^{n+1}M).
	$$
	The associativity formula yields $e_{r+1}(M/\mm^{n+1}M) = \sum_{\pp} \length_{B_\pp}(M_\pp) e_{r+1}(B/\pp)$ where $\pp$ runs through the minimal primes of $M$ such that $\pp \supseteq \mm$ and $\dim(B/\pp) = r+1$.
	Then the equality $m_r^0(M) = j_{r+1}(M)$ follows from \cite[Proposition 6.1.3]{FLENNER_O_CARROLL_VOGEL}.
	
	(ii) Throughout this part, we fix $v \gg 0$ big enough.
	We have the equalities 
	$$
	m^r_r(M) \;=\; \lim_{n \rightarrow \infty} \frac{P_{\sG_M^\star}(v,n)}{n^r/r!} \;=\; \lim_{n \rightarrow \infty} \frac{\length_{R}(M_v/\mm^{n+1}M_v)}{n^r/r!} \;=\; e_{r}(M_v).
	$$
	By utilizing the faithfully flat extension $R \rightarrow R[z]_{\mm R[z]}$, we may assume that the residue field $\kappa$ is infinite.
 	Then, by prime avoidance, we may choose an element $b \in B_v$ such that we obtain the short exact sequence of coherent $\OO_X$-modules $0 \rightarrow \widetilde{M} \rightarrow \widetilde{M}(v) \rightarrow (M/bM)^{\widesim{}}(v) \rightarrow 0$.
	Consequently, we have an exact sequence $0 \rightarrow \HH^0(X, \widetilde{M}) \rightarrow \HH^0(X, \widetilde{M}(v)) \rightarrow \HH^0(X, (M/bM)^{\widesim{}}(v))$ of finitely generated $R$-modules.
	Since $\dim(\HH^0(X, (M/bM)^{\widesim{}}(v))) < r$, it follows that $e_r(\HH^0(X, \widetilde{M})) = e_r(\HH^0(X, \widetilde{M}(\nu)))$.
	Thus the isomorphism $M_v \cong \HH^0(X, \widetilde{M}(v))$ (recall that $v \gg 0$) gives the desired equality $m_r^r(M) = e_r(\HH^0(X, \widetilde{M}))$.
	
	(iii)
	By prime avoidance, we may assume that $(0:_My)^{\widesim{}}=0$ and $\dim(\Supp((M/yM)^{\widesim{}})) < r$.
	Since $y \in B_+$ is general, we also have that $\iniTerm(y) \in [\sG]_{(1,0)}$ is general (see \autoref{rem_general_initial_form} below), and thus we may assume that $y$ is a $\sG$-parameter on $M$.
	Hence the equality $m_{r-1}^{i}(M/yM) = m_r^i(M)$ follows from \autoref{thm_hypersurface_sect}.
		
	(iv) The proof follows similarly to part (iii). 
	Here we need the assumption $\dim(\Supp((M/\mm M)^{\widesim{}})) < r$ to attain the condition $\dim(\Supp((M/x M)^{\widesim{}})) < r$.
\end{proof}

\begin{remark}
	\label{rem_general_initial_form}
	($\kappa$ infinite). 
	Consider the ideal $I = \left(\mm^\beta B_\alpha\right) \subset B$ for some $\alpha,\beta \in \NN$.
	Let $f_1,\ldots,f_e$ be a minimal generating set of $I$.
	By Nakayama's lemma, $\overline{f_1},\ldots,\overline{f_e}$ give a $\kappa$-basis of $I/(\mm B + B_+)I \cong \mm^\beta B_\alpha/\mm^{\beta+1}B_\alpha = \left[\sG\right]_{(\alpha,\beta)}$.
	Therefore, if $f = a_1f_1+\cdots+a_ef_e \in I_\alpha$ is a general element, then $\iniTerm(f) = \overline{a_1}\overline{f_1}+\cdots+\overline{a_e}\overline{f_e} \in \left[\sG\right]_{(\alpha,\beta)}$ is also a general element.
\end{remark}

\begin{remark}
	\label{rem_length_formula_y}
	($\kappa$ infinite).
	 By successively applying  \autoref{prop_gen_polar}(iii) and then utilizing \autoref{prop_gen_polar}(ii), we obtain the formula 
	$$
	m_r^i(M) \;=\; e_i\left(\HH^0\left(X,\, (M/(y_1,\ldots,y_{r-i})M)^{\widesim{}}\right)\right)
	$$
	with $y_1,\ldots,y_{r-i} \in B_+$ a sequence of general elements.
	In the next section, we also provide a length formula for polar multiplicities by cutting with general elements in $\mm \subset R$.
\end{remark}

Given a prime ideal $\pp \in \Spec(R)$, we have that $B_\pp = B \otimes_R R_\pp$ is a standard graded algebra over $R_\pp$ and that $M_\pp = M \otimes_R R_\pp$ is a finitely generated graded $B_\pp$-module. 
Then a natural question is to compare the polar multiplicities of $M$ and $M_\pp$.
The following lemma shows that polar multiplicities tend to not increase under localizations, thus mimicking the well-known behavior of Hilbert-Samuel multiplicities (see \cite[Theorem 40.1]{NAGATA}, \cite{BENNETT, LECH})).
Our proof below is inspired by \cite[Proposition 2.7]{PTUV}.

\begin{lemma}
	\label{lem_localiz_polar}
	Suppose that $R$ is equidimensional and catenary.
	Let $\pp \in \Spec(R)$ be a prime ideal and assume that $R/\pp$ is analytically unramified.
	Let $M$ be a finitely generated graded $B$-module.
	Set $r = \dim(\Supp(\widetilde{M}))$ and $r' = \dim(\Supp(\widetilde{M_\pp}))$, and assume that $\dim(R/\pp) = r-r'$.
	Then 
	$$
	m_r^i(M) \;\ge\; m_{r'}^{i-r+r'}(M_\pp).
	$$
\end{lemma}
\begin{proof}
	By making a purely transcendental residue field extension as in the proof of \autoref{prop_gen_polar}(ii), we may assume the $\kappa$ is not an algebraic extension of a finite field.
	Then \cite[Lemma 2.6]{PTUV} yields a sequence of elements $y_1,\ldots,y_{r-i} \in B_+$ which can be used to apply \autoref{rem_length_formula_y} to both $M$ and $M_\pp$.
	So, we obtain the equalities 
	$$
	m_r^i(M) \;=\; e_i\left(\HH^0\left(X,\, (M/(y_1,\ldots,y_{r-i})M)^{\widesim{}}\right)\right)
	$$
	and 
	$$
	m_{r'}^{i-r+r'}(M_\pp) \;=\; e_{i-r+r'}\left(\HH^0\left(X \times_{\Spec(R)} \Spec(R_\pp) ,\, (M_\pp/(y_1,\ldots,y_{r-i})M_\pp)^{\widesim{}}\right)\right).
	$$
	Let $N := \HH^0\left(X,\, (M/(y_1,\ldots,y_{r-i})M)^{\widesim{}}\right)$.
	By flat base change, we get $m_{r'}^{i-r+r'}(M_\pp) = e_{i-r+r'}(N_\pp)$.
	Let $\delta:=\dim(R)$.
	From the associativity formula and the fact that $R$ and $R_\pp$ are equidimensional and catenary, we obtain the equalities
	$$
	e_i(N) \;=\; \sum_{\substack{\qqq \in \Min(N)\\
	\dim(R/\qqq)=i}} \length_{R_\qqq}(N_\qqq)e(R/\qqq) \;=\; \sum_{\substack{\qqq \in \Min(N)\\
	\HT(\qqq)=\delta-i}} \length_{R_\qqq}(N_\qqq)e(R/\qqq)
	$$
	and 
	$$
	e_{i-r+r'}(N_\pp) = \sum_{\substack{\qqq \in \Min(N), \qqq \subseteq \pp \\
		\dim((R/\qqq)_\pp)=i-r+r'}} \length_{R_\qqq}(N_\qqq)e((R/\qqq)_\pp)
	= \sum_{\substack{\qqq \in \Min(N), \qqq \subseteq \pp \\
			\HT(\qqq)=\delta-i}} \length_{R_\qqq}(N_\qqq)e((R/\qqq)_\pp).
	$$
	To conclude the proof, we only need to show that $e(R/\qqq) \ge e((R/\qqq)_\pp)$ whenever $\qqq \subseteq \pp$.
	Since $\HT(\pp/\qqq) + \dim(R/\pp) = \dim(R/\qqq)$ and $R/\pp$ is analytically unramified, we obtain $e(R/\qqq) \ge e((R/\qqq)_\pp)$ from \cite[Theorem 40.1]{NAGATA} (also, see \cite{BENNETT, LECH}).
\end{proof}

\subsection{Polar multiplicities with respect to an ideal of linear forms}\hfill

In this subsection, we define and study a notion of polar multiplicities with respect to an ideal of linear forms in $B$.
Our results in this subsection profited from \cite{UV_CRIT_MOD}.
Let $I \subset B$ be an ideal generated by linear forms in $B$.
Let $M$ be a finitely generated graded $B$-module, and set $r = \dim(\Supp(\widetilde{M}))$.
We consider $\Rees^+(I) = B[IT, T^{-1}]$ with an ``internal grading''. 
This means that $\deg(T) = 0$ and that $\Rees^+(I)$ is a standard graded algebra over $R[T^{-1}]$ since $I \subset B$ is generated by linear forms.
We then have that $G:=\gr_I(\sB) \cong \Rees^+(I) / T^{-1} \Rees^+(I) = \bigoplus_{k=0}^\infty I^k/I^{k+1}$ becomes a standard graded $R$-algebra.
Furthermore, $\gr_I(M)$ is a finitely generated graded module over $G$ and the corresponding $v$-graded part is given by
$$
\big[\gr_I(M)\big]_v \;=\; \bigoplus_{k=0}^\infty \left[I^kM\big/ I^{k+1}M\right]_v.
$$
Here, our main interest is in the following invariants.

\begin{definition}
	We set $m_r^i(I, M) := m_r^i\left(\gr_I(M)\right)$ for all $0 \le i \le r$.
\end{definition}

The next proposition shows that these polar multiplicities are also additive.

\begin{proposition}
	\label{prop_add_polar_ideal}
	Let $I \subset \sB$ be an ideal generated by linear forms.
	Let $0 \rightarrow M' \rightarrow M \rightarrow M'' \rightarrow 0$ be a short exact sequence of finitely generated graded $\sB$-modules. 
	Set $r = \dim(\Supp(\widetilde{M}))$.
	Then we have the equality 
	$$
	m_r^i(I, M) \;=\; m_r^i(I, M') + m_r^i(I, M'')
	$$
	for all $0 \le i \le r$.
\end{proposition}
\begin{proof}
	We follow a similar approach to \autoref{thm_additive_polar}.
	Let $N := \Ker\left(\Rees^+(I; M) \rightarrow \Rees^+(I; M'')\right)$ and $L := N / \Rees^+(I; M')$.
	Following the steps that yielded \autoref{eq_ex_seq_N_gr_gr}, \autoref{eq_ex_seq_mult_L} and \autoref{eq_ex_seq_gr_A_N}, we now have the short exact sequences
	\begin{align*}
		&0 \longrightarrow N/T^{-1}N \longrightarrow \gr_I(M) \longrightarrow \gr_I(M'') \longrightarrow 0, \\
		&0 \longrightarrow U \longrightarrow L \xrightarrow{\quad T^{-1} \quad} L \longrightarrow V \longrightarrow 0, \\
		&0 \longrightarrow U \longrightarrow \gr_I(M') \longrightarrow N/T^{-1}N \longrightarrow V \longrightarrow 0.
	\end{align*}
 	We have that $\dim(\Supp(\widetilde{L})) \le r$.
	Then applying \autoref{thm_additive_polar} gives the equality
	\[
	m_r^i\left(\gr_I(M)\right) = m_r^i\left(\gr_I(M')\right) + m_r^i\left(\gr_I(M'')\right).
	\]
	This concludes the proof.
\end{proof}

The following theorem shows that the polar multiplicities considered in this subsection do not change when passing to a reduction.

\begin{theorem}
	\label{thm_reduction_ideals}
	Assume \autoref{setup_basic_polar}.
	Let $M$ be a finitely generated graded $\sB$-module and $I \subseteq J$ be $\sB$-ideals generated by linear forms.
	Set $r = \dim(\Supp(\widetilde{M}))$.
	If $I$ is a reduction of $J$ on $M$, then $m_r^i(I, M) = m_r^i(J, M)$ for all $0 \le i \le r$.
\end{theorem}
\begin{proof}
	By assumption, we have $I\cdot J^{k}M = J^{k+1}M = J \cdot J^kM$ for some positive integer $k \ge 1$.
	Thus we get $\gr_{I}(J^kM) = \gr_J(J^kM)$.
	For each $j \ge 0$, notice that $\gr_{I}\left(J^jM/J^{j+1}M\right) \cong J^jM/J^{j+1}M \cong \gr_{J}\left(J^jM/J^{j+1}M\right)$.
	We have the short exact sequences 
	\[
	0 \rightarrow J^kM \rightarrow M \rightarrow M/J^kM \rightarrow 0
	\]
	and 
	\[
	0 \rightarrow J^jM/J^{j+1}M \rightarrow M/J^{j+1}M \rightarrow M/J^jM \rightarrow 0
	\]
	for all $j \ge 0$.
	Hence by applying \autoref{prop_add_polar_ideal}, we obtain the equalities
	\begin{align*}
		m_r^i(I, M) &= m_r^i\left(I, J^kM\right) + m_r^i\left(I, M/J^kM\right) \\
		&= m_r^i\left(I, J^kM\right) + \sum_{j=0}^{k-1}m_r^i\left(I, J^jM/J^{j+1}M\right)\\
		&= m_r^i\left(J, J^kM\right) + \sum_{j=0}^{k-1}m_r^i\left(J, J^jM/J^{j+1}M\right)\\
		&= m_r^i\left(J, J^kM\right) + m_r^i\left(J, M/J^kM\right) \\
		&= m_r^i(J, M)
	\end{align*}
	that settle the claim of the theorem.
\end{proof}

\section{A length formula for polar multiplicities}
\label{sect_length_formula}

In this section, we provide a length formula for polar multiplicities by utilizing general elements in the maximal ideal of the local ring $R$ (see \autoref{rem_length_formula_y}). 
For that purpose, we can use the (local) deformation to the normal cone of Achilles and Manaresi \cite{ACHILLES_MANARESI_MULT}.
Here we continue using \autoref{setup_basic_polar}, but we now fix the following setup with further data.

\begin{setup}
	\label{setup_regular_elements}
	Assume \autoref{setup_basic_polar} with $\kappa$ an infinite field.
	Recall that $X = \Proj(B)$.
	Let $M$ be a finitely generated graded $B$-module, and set $r := \dim(\Supp(\widetilde{M}))$.
	Let $\sG:=\gr_{\mm B}(B)$ and $\sG_+ := \bigoplus_{n \ge 1} [\sG]_{n} = \bigoplus_{n \ge 1} \mm^nM/\mm^{n+1}M$.
	Let $x_0 := 0 \in B$ and $f_0:=0 \in \sG$.
	Let $\xx := (x_1,\ldots,x_r)$ be a sequence of general elements in $\mm$ and $\ff := (f_1=\iniTerm(x_1), \ldots,f_r=\iniTerm(x_r))$ be the corresponding sequence of initial forms in $[\sG]_{(0,1)} \cong \mm/\mm^2$.
	(One can also work in terms of filter-regular sequences as in \cite{ACHILLES_MANARESI_MULT}.)
\end{setup}

For a submodule $N \subseteq M$, the corresponding initial submodule in $\sG_M = \gr_\mm(M)$ is given by  
$$
\init(N) \;:=\; \bigoplus_{n=0}^\infty \left(N \cap \mm^nM + \mm^{n+1}M\right) \big/ \mm^{n+1}M \;\subset\; \sG_M.
$$
We utilize two \emph{intersection algorithms} inspired by the work of Achilles and Manaresi \cite{ACHILLES_MANARESI_MULT}.
These algorithms yield certain St\"uckrad-Vogel cycles that will be used to compute polar multiplicities.

\begin{definition}[St\"uckrad-Vogel cycle of $M$ with respect to $\xx$]
	Set $N_{-1} := 0 \subset M$.
	Inductively, we set $N_i := \left(N_{i-1} + x_iM\right) :_M \mm^\infty \subset M$ for $0 \le i \le r$.
	We have the cycle 
	$$
	\nu_i(\xx; M) \;:=\; \sum_\pp \length_{\sB_\pp}\left(M_\pp \big/ \left(N_{i-1} + x_iM\right)_\pp \right) \big[\Proj(B/\pp)\big] \; \in \; Z_{r-i}\big(X \times_{\Spec(R)} \Spec(\kappa)\big),
	$$
	where the sum runs over the  minimal primes of $M \big/ \left(N_{i-1} + x_iM\right)$ of dimension $r-i+1$ that contain $\mm M$.
\end{definition}

Via the natural maps $\sB \twoheadrightarrow \sB/\mm \sB \cong  [\sG]_{(*,0)} \hookrightarrow \sG$, the contraction of a prime ideal $\fP \subset \sG$ containing $\sG_+$ yields a prime ideal $\pp = \fP \cap \sB \subset \sB$  containing $\mm\sB$ and that satisfies $\sG/\fP \cong B/\pp$.
As a consequence, a prime ideal $\fP \subset \sG$ containing $\sG_+$ and contracting to a homogeneous prime ideal $\pp = \fP \cap B \subset B$ automatically gives a cycle in $X \times_{\Spec(R)} \Spec(\kappa)$.

\begin{definition}[St\"uckrad-Vogel cycle of $\sG_M$ with respect to $\ff$]
	Set $L_{-1} := 0 \subset \sG_M$.
	Inductively, we set $L_i := \left(L_{i-1} + f_i\sG_M\right) :_{\sG_M} \sG_+^\infty \subset \sG_M$ for $0 \le i \le r$.
	We have the cycle 
	$$
	\nu_i(\ff; \sG_M) \;:=\; \sum_\fP \length_{\sG_\fP}\left({(\sG_M)}_\fP \big/ \left(L_{i-1} + f_i\sG_M\right)_\fP \right) \big[\Proj(\sG/\fP)\big] \; \in \; Z_{r-i}\big(X \times_{\Spec(R)} \Spec(\kappa)\big),
	$$
	where the sum runs over the minimal primes of $\sG_M \big/ \left(L_{i-1} + f_i\sG_M\right)$ of dimension $r-i+1$ that contain $\sG_+$.
\end{definition}

\begin{remark}
	\label{rem_heights_Ji_Li}
    By prime avoidance, we can inductively assume that $x_i$ does not belong to any associated prime  $\pp \in \Ass_B(M/(x_1,\ldots,x_{i-1})M)$ such that $\pp \not\supseteq \mm B$ (i.e., $x_i \in \mm$ is filter-regular on $M/(x_1,\ldots,x_{i-1})$; see \cite[Appendix]{STUCKRAD_VOGEL_BUCHSBAUM_RINGS}, \cite[Propositions 1.5.11, 1.5.12]{BRUNS_HERZOG}).
	So, for any associated prime $\pp \in \Spec(B)$ of $M/N_i \cong M / (x_1,\ldots,x_i)M:_M \mm^\infty$, we have that $x_1,\ldots,x_i$ is a regular sequence on $M_\pp$.
	By verbatim arguments, for any associated prime $\fP \in \Spec(\sG)$ of $\sG/L_i$, we have that $f_1,\ldots,f_i$ is a regular sequence on $(\sG_M)_\fP$.
\end{remark}

The next lemma deals with the behavior of $\sG_M$ when cutting with a sequence of general elements in $\mm$.

\begin{lemma}
	\label{lem_cut_ass_Gr}
	The following statements hold: 
	\begin{enumerate}[\rm (i)]
		\item There are a natural surjective homomorphism 
		$$
		\sG_M\big/\left(f_1,\ldots,f_i\right)\sG_M \;\surjects\; \gr_\mm\big(M\big/(x_1,\ldots,x_i)M\big) \;\surjects\; \gr_\mm\big(M\big/(x_1,\ldots,x_i)M :_M \mm^\infty\big)
		$$
		which are isomorphisms for the graded parts $(v,n)$ with $n \gg 0$.
		\item We have the equalities 
		$
		L_i \;=\; \init\left(\left(x_1,\ldots,x_k\right)M\right) :_{\sG_M} \sG_+^{\infty} \;=\; \init\left(N_i\right) :_{\sG_M} \sG_+^{\infty}
		$ 
		of submodules of $\sG_M$. 
	\end{enumerate}
\end{lemma}
\begin{proof}
	(i) By induction, it suffices to consider  $i = 1$, hence set $x = x_1$ and $f = \iniTerm(x)$.
	The explicit description of the claimed surjections is as follows
	\begin{align*}
		\sG_M/f\sG_M \;\cong\;& \bigoplus_{n=0}^\infty\frac{\mm^n M}{\mm^{n+1}M + x\mm^{n-1}M} \\
		 &\;\;\surjects\;\; \gr_\mm(M/xM) \;\cong\; \bigoplus_{n=0}^\infty\frac{\mm^n M}{\mm^{n+1}M +  \mm^{n}M\cap xM}\\
		 &\qquad\qquad\qquad\;\;\surjects\;\; \gr_\mm(M/(xM:_M \mm^\infty)) \;\cong\; \bigoplus_{n=0}^\infty\frac{\mm^n M}{\mm^{n+1}M +  \mm^{n}M\cap (xM:_M \mm^\infty)}.
	\end{align*}
	Since $(xM:_M \mm^\infty)$ is obtained by removing primary components of $xM \subset M$ that contain $\mm^k M$ for some $k > 0$ (see, e.g., \cite[\S 3.6]{EISEN_COMM}), it follows that $\mm^{n}M\cap (xM:_M \mm^\infty) = \mm^{n}M\cap xM$ for $n \gg 0$.	
	Hence to conclude the proof it suffices to show the equality $x\mm^{n-1}M = \mm^{n}M\cap xM$ for $n \gg 0$.
	
	For the rest of the proof, let $n \gg 0$.
	We may assume $f$ is filter-regular on $\sG$ with respect to $\sG_+$, and consequently we obtain that $\left[\left(0 :_{\sG_M} f\right)\right]_{(*,n)} = 0$ (see, e.g., \cite[Lemmas 3.6, 3.7]{MIXED_MULT}).
	This translates into 
	$$
		\mm^nM \,\cap\, (\mm^{n+2}M :_M x) \; = \; \mm^{n+1}M,
	$$
	and then we get
	$
	\mm^{n+2}M  \,\cap\, x\mm^nM \; = \; x\mm^{n+1}M.
	$
	Applying this latter equality yields 
	\begin{equation}
		\label{eq_intersect_max_n1_n2}
		\mm^{n_1}M  \,\cap\, x\mm^{n_2}M \; = \; x\mm^{\max\{n_1-1,n_2\}}M
	\end{equation}
	for $n_1 \gg 0$ and $n_2 \gg 0$.
	The Artin-Rees lemma gives a positive integers $n_0 >0$ such that 
	\begin{equation}
		\label{eq_Artin_Rees_gr}
		\mm^nM \,\cap\, xM \;\subseteq\;   x\mm^{n-n_0}M.
	\end{equation}
	By combining \autoref{eq_intersect_max_n1_n2} and \autoref{eq_Artin_Rees_gr}, we obtain the claimed equality 
	$$
	\mm^nM \,\cap\, xM \;=\;  \mm^nM \,\cap\, x\mm^{n-n_0}M \; = \; x\mm^{n-1}M.
	$$
	Therefore the proof of part (i) is complete.
	
	(ii) It is a direct consequence of part (i).
\end{proof}

We are now ready for the main result of this section.
The proof closely follows the arguments of Achilles and Manaresi \cite[Theorem 3.3]{ACHILLES_MANARESI_MULT}.

\begin{theorem}[Deformation to the normal cone]
	\label{thm_def_normal_cone}
	Assume \autoref{setup_regular_elements}. We have 
	$
	\nu_i(\xx; M) = \nu_i(\ff; \sG_M)
	$
	for all $0 \le i \le r$.
\end{theorem}
\begin{proof}
	Let $\fP \in \Spec(\sG)$ be a prime ideal containing $\sG_+$ and $\pp = \fP \cap \sB \in \Spec(\sB)$ be its contraction.
	Recall that $\sG / \fP \cong \sB/\pp$ and $\pp \supseteq \mm B$. 
	From \cite[Proposition 1.5.15]{BRUNS_HERZOG} and \autoref{lem_cut_ass_Gr}, we obtain the equivalences 
	$$
	\left(\sG_M/L_{i-1}\right)_\fP = 0 \;\;\;\Leftrightarrow\;\;\;  \sG_M/L_{i-1} \otimes_\sB \sB_\pp = 0 \;\;\;\Leftrightarrow\;\;\; (\sG_M)_\pp / \left(\init(N_{i-1})_\pp :_{(\sG_M)_\pp} (\sG_\pp)_+^\infty\right) = 0.
	$$
	It is then clear that $\pp \not\in \Supp_B(M/N_{i-1})$ implies $\fP \not\in \Supp_{\sG}(\sG_M/L_{i-1})$.
	On the other hand, if we have $\fP \not\in \Supp_{\sG}(\sG_M/L_{i-1})$, then $\big[\gr_{\mm}(M_\pp/(N_{i-1})_\pp)\big]_{n} = 0$ for $n \gg 0$, however this can only happen if $(N_{i-1})_\pp \supset \mm^n M_\pp$ for $n \gg 0$.
	Since $\HL^0(M/N_{i-1}) = 0$, the condition $(N_{i-1})_\pp \supset \mm^n M_\pp$ yields the vanishing $(M/N_{i-1})_\pp = 0$.
	Therefore, it follows that $\fP  \in \Supp_{\sG}(\sG_M / (L_{i-1} + f_i\sG))$ if and only if $\pp \in \Supp_B(M/(N_{i-1} + x_iM))$.
	To conclude the proof, it suffices to show that the respective lengths appearing as coefficients in the cycles $\nu_i(\xx; M)$ and $\nu_i(\ff; \sG_M)$ are equal. 
	
	Now, suppose that $\fP$ appears in $\nu_i(\ff; \sG_M)$ and that, equivalently, $\pp$ appears in $\nu_i(\xx; M)$ (in particular,  $\dim(\sG/\fP) = \dim(\sB/\pp) = r-i+1$).
	Let $\overline{\sG_M} = \sG_M/L_{i-1}$ and $\overline{M} = M/N_{i-1}$.
	Then the modules $(\overline{\sG_M})_\fP$ and $\overline{M}_\pp$ are one-dimensional (see \autoref{rem_heights_Ji_Li}), and since by prime avoidance $f_i$ and $x_i$ are nonzerodivisor on $(\overline{\sG_M})_\fP$ and on $\overline{M}_\pp$, respectively, we obtain the equalities
	$$
	\length_{\sG_\fP}\left(
	\frac{(\sG_M)_\fP}{\left(L_{i-1} + f_i\sG_M\right)_\fP} \right) = e\left(f_i, (\overline{\sG}_M)_\fP\right) \quad \text{and} \quad \length_{R_\pp}\left(\frac{M_\pp}{\left(N_{i-1} + x_iM\right)_\pp }\right) = e\left(x_i, \overline{M}_\pp\right)
	$$
	expressing the lengths as multiplicities (see, e.g., \cite[Corollary 1.2.14]{FLENNER_O_CARROLL_VOGEL}).
	Again, by \cite[Proposition 1.5.15]{BRUNS_HERZOG},  $e(f_i, (\overline{\sG_M})_\fP) = e(f_i, (\overline{\sG_M})_\pp)$.
	We have that $\big[(\overline{\sG}_M)_\pp\big]_{n+1} = f_i \big[(\overline{\sG_M})_\pp\big]_{n}$ for $n \gg 0$.
	In other words, $f_i\sG_\pp$ is a reduction of $[\sG_\pp]_+$ on $(\overline{\sG_M})_\pp$, and then it follows that $x_iB_\pp$ is a reduction of $\mm B_\pp$ on $\overline{M}_\pp$.
	Consequently, we get the equalities 
	$$
	e\left(f_i, (\overline{\sG}_M)_\pp\right) = e\left([\sG_\pp]_+, (\overline{\sG_M})_\pp\right) \quad \text{ and } \quad e(x_i, \overline{M}_\pp) = e\left(\mm B_\pp, \overline{M}_\pp\right);
	$$
	see, e.g., \cite[Theorem 14.3]{MATSUMURA}.
	Due to \autoref{lem_cut_ass_Gr}, we get $e\left([\sG_\pp]_+, (\overline{\sG_M})_\pp\right) = e\left([\sG_\pp]_+, \gr_\mm(\overline{M}_\pp)\right)$.
	Finally, by definition, we have  $e\left([\sG_\pp)]_+, \gr_\mm(\overline{M}_\pp)\right) = e\left(\mm B_\pp, \overline{M}_\pp\right)$.
	This establishes the required equality 
	$$
	\length_{\sG_\fP}\left(
	\frac{(\sG_M)_\fP}{\left(L_{i-1} + f_i\sG_M\right)_\fP} \right)  \;= \; \length_{B_\pp}\left(\frac{M_\pp}{\left(N_{i-1} + x_iM\right)_\pp }\right)
	$$
	and finishes the proof.
\end{proof}

Next we show that the polar multiplicities of $M$ can be read from the cycles  $\nu_i(\ff; \sG_M)$.

\begin{lemma}
	\label{lem_polar_degree_GG}
	For all $0 \le i \le r$, we have $m_r^i(M) = \deg(\nu_i(\ff; \sG_M))$.
\end{lemma}
\begin{proof}
	We proceed inductively on $i$.
	Set $i = 0$.
	We may fix $n \gg 0$ big enough so that we have the equalities 
	$$
	m_r^0(M) =  \lim_{v \rightarrow \infty} \frac{P_{\sG_M^\star}(v, n)}{v^{r}/r!} =  \lim_{v \rightarrow \infty} \frac{\length_R\left(\big[\sG_M/\sG_+^{n+1}\sG_M\big]_{(v,*)}\right)}{v^{r}/r!}  = e_{r+1}\big( \sG_M/\sG_+^{n+1}\sG_M\big),
	$$
	here we see $\sG/\sG_+^{n+1} = \bigoplus_{v \ge 0} [\sG/\sG_+^{n+1}]_{(v,*)}$ as a standard graded algebra over the Artinian local ring $[\sG/\sG_+^{n+1}]_{(0,*)} = \gr_{\mm}(R) / [\gr_{\mm}(R)]_+^{n+1}$.
	From the associativity formula for multiplicities, we obtain 
	$$
	e_{r+1}\big(\sG_M/\sG_+^{n+1}\sG_M\big) \;=\; \sum_\fP \length_{\sG_\fP}\big((\sG_M)_\fP\big)\, e_{r+1}(\sG/\fP)
	$$
	where the sum runs over the minimal primes of $\sG_M$ of dimension $r+1$ that contain $\sG_+$ (we may assume $(\sG_+^{n+1}\sG_M)_\fP= 0$ since we are choosing $n\gg$ big enough).
	So, we have $m_r^0(M) = \deg(\nu_0(\ff; \sG_M))$.
	
	For $i\ge 1$, \autoref{lem_basic_polar} yields $m_r^i(M) = e\left(r-i,i-1; \sG_M/(0:_{\sG_M} \sG_+^\infty)\right)$.
	By prime avoidance, $f_1$ is a nonzerodivisor on $\sG_M/(0:_{\sG_M} \sG_+^\infty)$, and so 
	$$
	\Hilb_{\sG_M/\left((0:_{\sG_M} \sG_+^\infty) + f_1\sG_M\right)}(t_1,t_2) \;=\; \frac{1}{1-t_2}\Hilb_{\sG_M/(0:_{\sG_M} \sG_+^\infty)}(t_1,t_2).
	$$
	Set $\overline{\sG_M} := \sG_M/\left((0:_{\sG_M} \sG_+^\infty) + f_1\sG_M\right)$ and $\overline{\sG_M}^\star := \overline{\sG_M}[t^\star]$ with $t^\star$ again a variable of degree $(0,1)$.
	From \cite[Theorem A]{MIXED_MULT}, we conclude that $m_r^{i}(M) = e(r-i,i-1; \sG_M/(0:_{\sG_M} \sG_+^\infty)) = e(r-i,i-1; \overline{\sG_M}^\star)$ for all $i \ge 1$.
	By applying the argument of the above paragraph to $\overline{\sG_M}$ instead of $\sG_M$, we obtain  the equality $e(r-1,0; \overline{\sG_M}^\star)  = \deg(\nu_{1}(\ff; \sG_M))$.
	Therefore by repeating this whole process we eventually get the result. 
\end{proof}

We can now provide the promised length formula for polar multiplicities.

\begin{theorem}[Length formula for polar multiplicities]
	Assume \autoref{setup_regular_elements}.
	For all $0 \le i \le r$, we have 
	$$
	m_r^i(M) \;=\; \sum_{\pp} 
\length_{\sB_\pp}\left(
\frac{M_\pp}
{\big(\left((x_1,\ldots,x_{i-1})M :_M  \mm^\infty\right) + x_iM\big)_\pp} \right)\cdot e(\sB/\pp)
	$$
	where the sum runs through the minimal primes of $M/\left((x_1,\ldots,x_{i-1})M:_M\mm^\infty+x_iM\right)$ of dimension $r-i+1$ that contain $\mm B$.
	Here we use the convention that $(x_1,\ldots,x_{i-1})M :_M \mm^\infty$ is $0$ for $i = 0$.
\end{theorem}
\begin{proof}[First proof]
	By combining \autoref{thm_def_normal_cone} and \autoref{lem_polar_degree_GG}, we get $m_r^i(M) = \deg(\nu_i(\xx; M))$.
	On the other hand, the displayed formula is the degree of the cycle $\nu_i(\xx; M)$.
\end{proof}
\begin{proof}[Second proof]
 	Let $Q_i := M/\left((x_1,\ldots,x_{i-1})M:_M\mm^\infty+x_iM\right)$.
 	By \cite[Proposition 6.1.3]{FLENNER_O_CARROLL_VOGEL}, it suffices to show that $m_r^i(M) = j_{r-i+1}(Q_i)$.
 	Let $\overline{Q_i} := Q_i / \HL^0(Q_i)$ and notice that $Q_{i+1} \cong \overline{Q_{i}}/x_{i+1}\overline{Q_i}$.
 	Then \autoref{prop_gen_polar}(iv) yields the equality $m_{r-i-1}^j(Q_{i+1}) = m_{r-i}^{j+1}(\overline{Q_i})$ for all $0 \le j \le r-i-1$.
 	The short exact sequence 
 	$$
 	0 \rightarrow \HL^0(Q_i) \rightarrow Q_i \rightarrow \overline{Q_i} \rightarrow 0
 	$$
 	 and \autoref{thm_additive_polar} give the equality $m_{r-i}^j(Q_i) = m_{r-i}^j(\overline{Q_i}) + m_{r-i}^j(\HL^0(Q_i))$ for all $0 \le j \le r-i$.
 	 By \autoref{prop_gen_polar}(i), we have $m_{r-i}^0(\overline{Q_i})=0$.
 	 Since $\mm^n \cdot \HL^0(Q_i) =0$ for $n \gg 0$, \autoref{lem_basic_polar} yields the equalities $m_{r-i}^j(\HL^0(Q_i)) = e\big(r-i-j,j; \sG_{\HL^0(Q_i)}^\star\big) = e\big(r-i-j,j-1; \sG_{\HL^0(Q_i)}\big) = 0$ for all $1 \le j \le r-i$.
 	 Therefore, we can iterate this process and obtain the equalities
 	 $$
 	 m_r^i(M) \;=\; m_{r-i}^0(Q_i) \;=\; j_{r-i+1}(Q_i),
 	 $$ 
 	 where the last one follows from \autoref{prop_gen_polar}(i).
\end{proof}

\section{Polar multiplicities: integrality and birationality}
\label{sect_criteria}

In this section, we give criteria for the integrality and birationality of an inclusion of graded algebras in terms of polar multiplicities. 
Here the following setup is fixed. 

\begin{setup}
	\label{setup_crit_polar}
	Let $A \subseteq B$ be a homogeneous inclusion of standard graded algebra over a Noetherian local ring $R = A_0 = B_0$.
	Consider the associated morphism $f : U \subseteq X = \Proj(B) \rightarrow Y=\Proj(A)$ where $U = X \setminus V_+(A_+B)$.
	Set $r = \dim(X)$.
\end{setup}

A very important observation of Simis, Ulrich and Vasconcelos \cite{SUV_MULT} is that the integrality and birationality of the inclusion $A \hookrightarrow B$ can be studied via the intermediate algebra $\gr_{A_+B}(B)$.
As in \cite[\S 3]{SUV_MULT}, we consider the Rees algebra $\Rees(A_+B) = B[A_+BT]  \subset B[T]$ and the extended Rees algebra $\Rees^+(A_+B) = B[A_+BT, T^{-1}]  \subset B[T, T^{-1}]$ with standard $\NN$-gradings (we can simply set $\deg(T) = 0$ and utilize the induced grading).
Thus, we see $\Rees(A_+B)$ as a standard graded $R$-algebra and $\Rees^+(A_+B)$ as a standard graded $R[T^{-1}]$-algebra.
As a consequence, we also see the associated graded ring $G := \gr_{A_+B}(B) \cong \Rees^+(A_+B)/T^{-1}\Rees^+(A_+B)$ as a standard graded $R$-algebra.
Let $t \ge 1$ be a positive integer.
From \cite[Proposition 3.2]{SUV_MULT}, the $v$-th graded part of the $G$-module $B_tG$ is given by 
$$
\left[\sB_tG\right]_v \;\cong \; \bigoplus_{j=1}^{v-t+1} \sA_{j-1}\sB_{v-j+1} \big/ \sA_{j}\sB_{v-j}.
$$
By applying the functor $- \otimes_R R/\mm^{n+1}$ to the short exact sequence 
\begin{equation*}
	0 \longrightarrow \sA_{j-1}\sB_{v-j+1} \big/ \sA_{j}\sB_{v-j}  \longrightarrow \sB_v \big/ \sA_{j}\sB_{v-j} \longrightarrow  \sB_v \big/ \sA_{j-1}\sB_{v-j+1} \longrightarrow  0,
\end{equation*}
we get the exact sequence 
$$
\frac{\sA_{j-1}\sB_{v-j+1}}{\sA_{j}\sB_{v-j} + \mm^{n+1}\sA_{j-1}\sB_{v-j+1}} \longrightarrow \frac{\sB_v}{\sA_{j}\sB_{v-j} + \mm^{n+1} \sB_v} \longrightarrow  \frac{\sB_v}{\sA_{j-1}\sB_{v-j+1} + \mm^{n+1} \sB_v} \longrightarrow  0.
$$
Thus summing up gives us the following
\begin{equation}
	\label{eq_ineq_gr_Q}
	P_{\sG_{B_tG}^\star}(v, n) = \length_R\big([B_tG]_v/\mm^{n+1}[B_tG]_v\big) \;\ge\; \length_R\big(B_v/(A_{v-t+1}B_{t-1}+\mm^{n+1}B_v)\big) 
\end{equation}
for all $v \gg 0$ and $n \gg 0$.
As it can be helpful, we provide  vanishing criteria for polar multiplicities.

\begin{lemma}
	\label{lem_vanish}
	We have $m_r^i(B) = 0$ if $i > \dim(R)$ or $i < r-\dim(\Proj(B/\mm B))$.
\end{lemma}
\begin{proof}
	By \autoref{lem_basic_polar}, we have $m_r^i(B) = e(r-i, i; \sG^\star)$ for all $0 \le i \le r$.
	 From \cite[Propostion 3.1]{POSITIVITY}, we obtain the vanishing $e(i,j;\sG^\star) = 0$ if $i > \dim\big(\Proj(\sG^\star/ ([\sG^\star]_{(0,1)}))\big) = \dim(\Proj(B/\mm B))$ or $j > \dim\big(\Proj(\sG^\star/ ([\sG^\star]_{(1,0)}))) = \dim(\gr_\mm(R)) = \dim(R)$.
\end{proof}

We consider the following polar multiplicities that will play an important role in the criteria that we provide. 

\begin{definition}
	We set $m_r^i(A, B) := m_r^i(A_+B, B) =  m_r^i(G)$  for all $0 \le i \le r$.
\end{definition}

The main result of this section is the  theorem below. 

\begin{theorem}
	\label{thm_criteria}
	Assume \autoref{setup_crit_polar}.
	The following statements hold: 
	\begin{enumerate}[\rm (i)]
		\item $m_r^i(A, B) \ge m_r^i(A)$ for all $0 \le i \le r$.
		\item If $R$ is equidimensional and catenary, then $m_r^i(A, B) \ge m_r^i(B)$ for all $0 \le i \le r$.
		\item {\rm(Integrality)} Consider the following conditions: 
		\begin{enumerate}[\rm (a)]
			\item A finite morphism $f : X \rightarrow Y$ is obtained.
			\item $m_r^i(A, B) = m_r^i(G/B_tG)$ for all $0 \le i \le r$ and $t \gg 0$.
			\item $m_r^i(A, B) = m_r^i(B)$ for all $0 \le i \le r$.
		\end{enumerate}
		In general, the implications {\rm(a)} $\Rightarrow$ {\rm(b)} and {\rm(a)} $\Rightarrow$ {\rm(c)} hold. 
		Moreover, if $B$ is equidimensional and catenary, then the reverse implication {\rm(b)} $\Rightarrow$ {\rm(a)} also holds.
		\item {\rm(Integrality $+$ Birationality)} If we obtain a finite birational morphism $f : X \rightarrow Y$, then $m_r^i(A, B) = m_r^i(A)$ for all $0 \le i \le r$.
		The converse holds if $B$ is equidimensional and catenary.
	\end{enumerate}
\end{theorem}
\begin{proof}
	(i) From the short exact sequence $0  \rightarrow B_1G \rightarrow G \rightarrow A \rightarrow 0$ and \autoref{thm_additive_polar}, we obtain the inequality $m_r^i(A, B) = m_r^i(G) \ge m_r^i(A)$.
	
	(ii)
	Consider the local ring $S := R[T^{-1}]_{(\mm, T^{-1})}$ and the standard graded $S$-algebra $C := \Rees^+(A_+B) \otimes_{R[T^{-1}]} S$.
	Then \autoref{thm_hypersurface_sect} and \autoref{lem_localiz_polar} yield the inequalities
	$$
	m_r^i(A, B) \;=\; m_r^i(G) \;\ge\; m_{r+1}^{i+1}(C) \;\ge\; m_{r}^{i}(C \otimes_S S_{\mm S}) \;=\; m_r^i(B),
	$$
	where the last equality follows from the isomorphism $C \otimes_S S_{\mm S} \cong B[T] \otimes_{R[T]} R[T]_{\mm R[T]}$. 
	So the result of this part follows.
	
	(iii) Suppose that $A \hookrightarrow B$ is an integral extension (i.e., $\sqrt{A_+B} \supseteq B_+$).
	Then, for $t \gg 0$, we have $B_t G = 0$ and so the short exact sequence $0 \rightarrow B_tG \rightarrow G \rightarrow G/B_tG \rightarrow 0$ and \autoref{thm_additive_polar} give the equality $m_r^i(A, B) = m_r^i(G) = m_r^i(G/B_tG)$.
	Since $A_+B$ is a reduction of $B_+$, \autoref{thm_reduction_ideals} implies that 
	$$
	m_r^i(A, B) \;=\; m_r^i(G) \;=\; m_r^i(\gr_{B_+}(B)) \;=\; m_r^i(B),
	$$
	where the last equality holds since $\gr_{B_+}(B) \cong B$.
	So, we have the implications (a) $\Rightarrow$ (b) and (a) $\Rightarrow$ (c).

	Suppose that $A \hookrightarrow B$ is not an integral extension (i.e., $\sqrt{A_+B} \not\supseteq B_+$) and that $B$ is equidimensional and catenary.
	Let $x \in X$ be a closed point with associated prime $P \subset B$ such that $P \supseteq A_+B$.
	By \cite[Lemmas 6.5, 6.6]{RELATIVE_MIXED}, we obtain $P \supseteq \mm B$, $d_x = \dim(\OO_x) = r$ and the positivity of the following limit 
	$$
	\lim_{n \rightarrow \infty} \frac{\length_{R}\left(B_{t+nd}/\left[P^n\right]_{t+nd}\right)}{n^r} \;>\; 0
	$$
	for large enough integers $t \gg 0$ and $d \gg 0$.
	Hence, from the inclusion $P \supseteq \mm B + A_+B$ and equation \autoref{eq_ineq_gr_Q}, we get the positivity of the following limits
	$$
	\lim_{n \rightarrow \infty} \frac{P_{\sG_{B_{t+1}G}^\star}(t+nd, n-1)}{n^r} \;\ge\;
	\lim_{n \rightarrow \infty} \frac{\length_{R}\left(B_{t+nd}/\left(A_{nd}B_t + \mm^nB_{t+nd}\right)\right)}{n^r} \;>\; 0.
	$$
	Consequently,  \autoref{lem_basic_polar} yields 
	$$
	\sum_{i=0}^r \,\frac{d^{r-i}}{i!(r-i)!}\, m_r^i(B_{t+1}G) \;>\; 0,
	$$
	and so we should have $m_r^i(B_{t+1}G) > 0$ for some $0 \le i \le r$.
	Finally, since we have the equation $m_r^i(G) = m_r^i(G/B_{t+1}G) + m_r^i(B_{t+1}G)$ by \autoref{thm_additive_polar}, we obtain the implication (b) $\Rightarrow$ (a) (under our current assumption that $B$ is equidimensional and catenary).
	So the proof of part (iii) is complete.
	
	(iv) Since $m_r^i(A) = m_r^i(G/B_1G) \le m_r^i(G/B_tG)$ for all $t \ge 1$ and we already settled part (iii), in this part we may assume that $A \hookrightarrow B$ is an integral extension.
	Moreover, by part (iii), we may assume $m_r^i(A, B) = m_r^i(B)$.
	Then \autoref{cor_associative_polar} gives the equalities 
	$$
	m_r^i(A) = \sum_\pp \length_{A_\pp}(A_\pp) m_r^i(A/\pp) \quad \text{ and } \quad m_r^i(B) = \sum_\pp \length_{A_\pp}(B_\pp) m_r^i(A/\pp)
	$$
	where both sums run through the  minimal primes of $A$ of dimension $r+1$.
	By \autoref{lem_dim_sG_star}, for any relevant minimal prime $\pp$ of $A$ of dimension $r+1$, there is some $0 \le i \le r$ such that $m_r^i(A/\pp) > 0$.
	Therefore, under our current assumptions, we obtain that $f : X \rightarrow Y$ is birational if and only if $m_r^i(A) = m_r^i(B)$ for all $0 \le i \le r$.
	This completes the proof of the theorem.
\end{proof}

\section{Mixed Buchsbaum-Rim multiplicities and reductions of modules}
\label{sect_modules}

In this short section, by applying the result of \autoref{thm_criteria}, we obtain a criterion for integral dependence of modules in terms of certain mixed Buchsbaum-Rim multiplicities.

Let $R$ be a Noetherian local ring of dimension $d$. 
Let $E$ be a finitely generated $R$-module having rank $e \ge 1$.
Following \cite{SUV_REES_MOD}, we say that the \emph{Rees algebra} $\Rees(E)$ is given as the symmetric algebra $\Sym(E)$ modulo its $R$-torsion.
Also, we have the equality $\dim(\Rees(E)) = d+e$ (see \cite[Proposition 2.2]{SUV_REES_MOD}).
Set $r := d+e-1$.
We consider the following \emph{mixed Buchsbaum-Rim multiplicities} (see \cite{KLEIMAN_THORUP_GEOM, KLEIMAN_THORUP_MIXED}):
$$
\br_i(E) \;:=\; m_r^i\left(\Rees(E)\right) \quad \text{ for all } \quad 0 \le i \le r,
$$
that we define as polar multiplicities of the Rees algebra $\Rees(E)$.
Given an $R$-submodule $U \subseteq E$, we say that $U$ is a \emph{reduction} of $E$ if $E^{v+1} = UE^v$ for some $v \ge 0$ (equivalently, if $\Rees(U) \hookrightarrow \Rees(E)$ is an integral extension).
Given an $R$-submodule $U \subseteq E$ also having rank $e$, we consider the following invariant 
$$
\br_i(U, E) \;:=\; m_r^i\left(\Rees(U), \Rees(E)\right)   \quad \text{ for all } \quad 0 \le i \le r.
$$

\begin{theorem}
	\label{thm_reduction_mod}
	Let $R$ be an equidimensional and universally catenary Noetherian local ring of dimension $d$.
	Let $U \subseteq E$ be an inclusion of finitely generated $R$-modules having rank $e \ge 1$.
	Set $r = d+e-1$.
	Then the following two conditions are equivalent: 
	\begin{enumerate}[\rm (a)]
		\item $U$ is a reduction of $E$.
		\item $\br_i(U, E) = \br_i(E)$ for all $0 \le i \le r$.
	\end{enumerate}
\end{theorem}
\begin{proof}
	By applying \autoref{thm_criteria} to the inclusion of standard graded $R$-algebras $\Rees(U) \hookrightarrow \Rees(E)$, the result of the theorem follows.
\end{proof}

Finally, we should ask the following question. 

\begin{question}
 	Under which conditions do we have the equivalence that $U$ is a reduction of $E$ if and only if $\br_i(U) = \br_i(E)$ for all $i$?
\end{question}

\section*{Acknowledgments}
We thank the reviewer for carefully reading our paper and for their comments and corrections.
 
 \bibliography{references}

\end{document}